\documentclass{article}
\usepackage{amsmath,amsfonts,amssymb,amsthm}
\usepackage{xcolor}

\renewcommand{\le}{\leqslant}
\renewcommand{\ge}{\geqslant}
\renewcommand{\leq}{\leqslant}
\renewcommand{\geq}{\geqslant}
\renewcommand{\emptyset}{\varnothing}

\newcommand{\tmod}{\ \mathsf{mod}\ }

\newcommand{\ints}{\mathbb{Z}}
\newcommand{\natu}{\mathbb{N}}

\newcommand{\bsc}{\boldsymbol{c}}
\newcommand{\bsk}{\boldsymbol{k}}
\newcommand{\bsu}{\boldsymbol{u}}
\newcommand{\bsx}{\boldsymbol{x}}

\newcommand{\rd}{\,\mathrm{d}}
\newcommand{\dunif}{\mathbb{U}}
\newcommand{\mc}{\mathrm{MC}}
\newcommand{\qmc}{\mathrm{QMC}}
\newcommand{\rqmc}{\mathrm{RQMC}}
\newcommand{\simiid}{\stackrel{\mathrm{iid}}\sim}

\newcommand{\e}{\mathbb{E}}
\newcommand{\var}{\mathrm{var}}
\newcommand{\vol}{\mathrm{vol}}
\newcommand{\rank}{\mathrm{rank}}

\newcommand{\one}{\mathbf{1}}
\newcommand{\bvhk}{\mathrm{BVHK}}

\newcommand{\phm}{\phantom{-}}
\newcommand{\tran}{\mathsf{T}}

\newtheorem{theorem}{Theorem}
\newtheorem{lemma}{Lemma}
\newtheorem{corollary}{Corollary}
\newtheorem{proposition}{Proposition}
\theoremstyle{definition}
\newtheorem{definition}{Definition}
\newtheorem{remark}{Remark}
\newtheorem{example}{Example}

\author{Zexin Pan\\Stanford University
\and
Art B. Owen\\Stanford University}
\title{The nonzero gain coefficients of Sobol's sequences are
always powers of two}
\date{June 2021}
\begin{document}
\maketitle
\begin{abstract}
When a plain Monte Carlo estimate on $n$ samples
has variance $\sigma^2/n$, then scrambled digital
nets attain a variance that is $o(1/n)$ as $n\to\infty$.
For finite $n$ and an adversarially selected integrand,
the variance of a scrambled $(t,m,s)$-net can
be at most $\Gamma\sigma^2/n$ for a maximal
gain coefficient $\Gamma<\infty$.
The most widely used digital nets and sequences
are those of Sobol'.
It was previously known that $\Gamma\le 2^t3^s$
for Sobol' points as well as Niederreiter-Xing points.
In this paper we study nets in base $2$.
We show that $\Gamma \le2^{t+s-1}$ for nets.
This bound is a simple, but apparently unnoticed,
consequence of a microstructure analysis in
Niederreiter and Pirsic (2001).  We obtain 
a sharper bound that is smaller 
than this for some digital nets. 
We also show that all nonzero gain coefficients must be powers of two.  A consequence of this latter fact
is a simplified algorithm for computing gain
coefficients of nets in base $2$.
\end{abstract}

\section{Introduction}

Numerical integration is a fundamental task in scientific computation.
In high dimensional problems, Monte Carlo (MC) methods are widely
used for integration because they are less affected by dimension than
classical methods, such as those in \cite{davrab}.
The MC problems we study are to compute
$\mu =\int_{[0,1]^s}f(\bsx)\rd\bsx$
for a dimension $s\ge1$
and most of our attention is on $f\in L^2[0,1]^s$.
This $\mu $ is the mathematical
expectation $\e(f(\bsx))$ for $\bsx\sim \dunif[0,1]^s$.
By using transformations from \cite{devr:1986}
we can greatly expand MC to expectations of
quantities with non-uniform distributions over
domains other than the unit cube, and so for this
paper it suffices to work with $\bsx\sim\dunif[0,1]^s$.

The MC estimate of $\mu$ is
\begin{align}\label{eq:mc}
\hat\mu_{\mc} = \frac1n\sum_{i=0}^{n-1}f(\bsx_i),\quad\bsx_i\simiid\dunif[0,1]^s.
\end{align}
The independent uniform draws $\bsx_i$ will form clusters
and leave gaps in $[0,1]^s$. This fact has lead to the development
of quasi-Monte Carlo (QMC) methods, beginning
with \cite{rich:1951}, designed to cover the unit cube more
evenly. See \cite{dick:kuo:sloa:2013} for a recent survey.
A QMC estimate $\hat\mu_{\qmc}$ has the same form as
$\hat\mu_{\mc}$ from~\eqref{eq:mc} except that
$n$ distinct points $\bsx_i\in[0,1]^s$ are chosen deterministically
so as to make the discrete uniform distribution
on $\{\bsx_0,\dots,\bsx_{n-1}\}$ close to the continuous
uniform distribution on $[0,1]^s$, by minimizing a measure of
the discrepancy (see \cite{dick:pill:2014})
between those distributions.

Using the Koksma-Hlawka inequality \cite{hick:2014} it is possible
to show that some QMC constructions attain
\begin{align}\label{eq:koksmahlawka}
|\hat\mu_{\qmc}-\mu| =O( n^{-1}\log(n)^{s-1})
\end{align}
when $f$ has bounded variation in the sense of Hardy
and Krause, which we write as $f\in\bvhk=\bvhk[0,1]^s$.
See~\cite{variation} for a description of this variation.
A drawback of QMC points is that they do not support
a practical strategy to compute the bound
in~\eqref{eq:koksmahlawka}.
Randomized QMC (RQMC) points $\bsx_0,\dots,\bsx_{n-1}$
are constructed so that individually $\bsx_i\sim\dunif[0,1]^d$
while collectively these points
have the low discrepancy that makes~\eqref{eq:koksmahlawka} hold.
Then we can estimate our error statistically, using independent
replicates of the randomization procedure.
See~\cite{lecu:lemi:2002} for a survey of RQMC.

In this paper we focus on perhaps the most widely used
QMC method, the Sobol' sequences of \cite{sobo:1967}.
We consider randomizing them
with the RQMC method known as scrambled nets from \cite{rtms}.
The MC estimate satisfies
\begin{align}\label{eq:mcrate}
\e( (\hat\mu_{\mc}-\mu)^2)=\frac{\sigma^2}n.
\end{align}
Thus MC has a root mean squared error (RMSE) of $\sigma/n^{1/2}$.
The QMC error in~\eqref{eq:koksmahlawka} is asymptotically
better, but for large $s$ the $\log(n)^{s-1}$ factor leaves room for doubt
about QMC at feasible sample sizes.

For scrambled nets we have
\begin{align}\label{eq:gainbound}
\e( (\hat\mu_{\rqmc}-\mu)^2)\le \frac{\Gamma\sigma^2}n
\end{align}
for a maximal gain coefficient $\Gamma<\infty$, removing
the powers of $\log(n)$.
If $f\in\bvhk$, then~\eqref{eq:koksmahlawka} also holds
for $\hat\mu_{\rqmc}$, so RQMC gets the asymptotic 
benefit of QMC
while~\eqref{eq:gainbound} bounds how much
worse RQMC could be compared to MC for finite $n$
(with an adversarily chosen integrand).

When scrambling the nets taken from Faure 
sequences~\cite{faures}, it is known
from \cite{owensinum} that $\Gamma \le \exp(1)\doteq 
2.718$. The nets of Sobol'~\cite{sobo:1967} appear to be more widely
used.  For them it is known from \cite{snxs} that
$\Gamma\le 2^t3^s$ where $t$ is the quality parameter
that we describe below.
In this paper we improve that bound to show that
$\Gamma\le 2^{t+s-1}$. 
This bound can also be deduced from the
results of Niederreiter and Pirsic \cite{microstruct},
but to our knowledge this has not been remarked
on before. We further show that
all the nonzero gain coefficients are powers of two
and we provide a slight improvement in
the microstructure gain bounds from \cite{microstruct}.

An outline of this paper is as follows.
Section~\ref{sec:background} defines digital nets
and sequences, and reviews properties of scrambled
digital nets. Section~\ref{sec:bound} proves our
bound $2^{t+s-1}$.
Section~\ref{sec:pow2} proves that nonzero
gain coefficients must be powers of $2$.
Both of these results hold for any scrambled nets in
base $2$ including those of Sobol' \cite{sobo:1967}
as well as those of Niederreiter and Xing \cite{niedxing96, niedxing96b}
and base $2$ nets constructed via polynomial
lattice rules, as described in \cite{dick:pill:2010}.
Section~\ref{sec:reduced} shows that an improved
exponent of $2$ is possible by sharpening the
usual notion of the $t$ parameter for a subset
of variables. One concrete example with an
improved exponent is provided using
shift nets of Schmid \cite{schm:1998}.
Section~\ref{sec:disc} has a discussion.

\section{Notation and background}\label{sec:background}

In this section we define digital nets
and sequences.  Then we describe methods
of scrambling them and their properties.
The key property in this paper is the
set of gain coefficients of a digital net.

Throughout this paper we have a dimension $s\ge1$.
We write $1{:}s$ for $\{1,2,\dots,s\}$.
We use $\ints$ for the integers, $\natu_0$ for non-negative integers,
and for integers $n\ge1$ we let $\ints_n=\{0,1,\dots,n-1\}$.
For $u\subseteq 1{:}s$ and $\bsx=(x_1,\dots,x_s)\in[0,1]^s$
we write $\bsx_u$ for the tuple $(x_j)_{j\in u}$.
The cardinality of $u$ is written $|u|$.
For a set with a lengthy definition,
$\#$ may be used for cardinality.
For a statement $S$ we use $\one_S$ or $\one\{S\}$,
depending on readability, to denote
a variable that is $1$ when $S$ holds and $0$ otherwise.

\subsection{Digital nets and sequences}

We let $b\ge2$ be an integer base in which
to represent integers and points in $[0,1)$.
We work with half-open intervals because
we will need to partition $[0,1)^s$ into
congruent subsets.  Note that the problems
are still defined as $\int_{[0,1]^s}f(\bsx)\rd\bsx$
because QMC is strongly connected to Riemann 
integration \cite{nied:1978} and the notion of 
bounded variation
that we use is also defined on closed unit cubes.
We begin with some standard definitions.

\begin{definition}
An $s$-dimensional elementary interval in base $b$ has the form
$$
E(\bsk,\bsc)=\prod_{j=1}^s\Bigl[
\frac{c_j}{b^{k_j}},
\frac{c_j+1}{b^{k_j}}
\Bigr)
$$
where $\bsk = (k_1,\dots,k_s)\in\ints^s$
and $\bsc = (c_1,\dots,c_s)\in\ints^s$
satisfy $k_j\ge0$ and $0\le c_j <b^{k_j}$.
\end{definition}

Given $\bsk$, we define $|\bsk|=\sum_{j=1}^sk_j$.
For a given vector $\bsk$,
the $b^{|\bsk|}$ elementary intervals $E(\bsk,\bsc)$
partition $[0,1)^s$ into congruent sub-intervals.
Ideally they should all get the same number of our
integration points $\bsx_i$ and digital nets
defined next make this happen in certain cases.

\begin{definition}
For integers $m\ge t\ge0$ and $b\ge 2$ and $s\ge1$,
a $(t,m,s)$-net in base $b$ is a sequence
$\bsx_0,\dots,\bsx_{n-1}\in[0,1)^s$ for $n=b^m$ where
$$
\sum_{i=0}^{n-1}\one\{\bsx_i\in E(\bsk,\bsc)\} =
nb^{-|\bsk|} = b^{m-|\bsk|}
$$
for every elementary interval $E(\bsk,\bsc)$ with $|\bsk|\le m-t$.
\end{definition}

Other things being equal, we would prefer smaller
$t$ and $t=0$ is the best.
For a given $m$ and $s$ and $b$, the smallest attainable
$t$ might be larger than $0$.
The minT project \cite{schu:schm:2006,schu:schm:2009}
keeps track of the minimum achieved values of $t$
for given $m$ and $s$ and $b$ along with
known lower bounds.
When we refer to the value of $t$ for a
sequence of points, we mean the smallest value
of $t$ for which the sequence is a $(t,m,s)$-net.

\begin{definition}
For integers $t\ge0$ and $b\ge 2$ and $s\ge1$,
a $(t,s)$-sequence in base $b$ is an infinite sequence
$\bsx_0,\bsx_1,\dots\in[0,1)^s$
such that for any integer $m\ge t$ and any integer
$r\ge0$ the subsequence
$$
\bsx_{rb^m}, \bsx_{rb^m+1},\dots,\bsx_{(r+1)b^m-1}\in[0,1)^s
$$
is a $(t,m,s)$-net in base $b$.
\end{definition}

The value of $(t,s)$-sequences is that they
provide an extensible set of $(t,m,s)$-nets.
The first $b^m$ points are a $(t,m,s)$-net
in base $b$ and if we increase the
sample to $b^{m+1}$ points then we have
included $b-1$ more $(t,m,s)$-nets and
they're carefully constructed to fill the gaps
that each other leave, so that taken together
they now comprise a $(t,m+1,s)$-net in base $b$.
Taking $b$ of those $(t,m+1,s)$-nets yields
a $(t,m+2,s)$-net, and so on.
The $(t,m,s)$-nets that we study are
taken to be the first $b^m$ points of a $(t,s)$-sequence.

The first $(t,m,s)$-nets and $(t,s)$-sequences
are those of Sobol' \cite{sobo:1967}.
They are all in base $b=2$.
Sobol's construction actually defines a whole family
of point sequences, determined by one's choice of
`direction numbers'.
Joe and Kuo \cite{joe:kuo:2008} made an extensive
search for good direction numbers and their
choices are widely used.

The smallest value of $t$ that one can attain
is nondecreasing in~$s$.
The most favorable growth rates for $t$ as
a function of $s$ are in the $(t,s)$-sequences of
Niederreiter and Xing \cite{niedxing96, niedxing96b}.
These are ordinarily implemented in base $2$.

The $(t,s)$-sequences of
Faure \cite{faures} have $t=0$ but they require
a prime base $b\ge s$.
The modern notion of digital nets and sequences
is based on the synthesis in \cite{nied87}.
That reference also generalizes Faure's construction
to bases $b=p^r$ for a prime number $p$
and integer $r\ge1$.

If $\bsx_0,\dots,\bsx_{n-1}$ is a $(t,m,s)$-net in base $b$
then $\bsx_{0,u},\dots,\bsx_{n-1,u}\in[0,1)^{|u|}$
form a $(t,m,|u|)$-net in base $b$.
It is common that the quality parameter of these
projected digital nets is smaller than the
one for the original net.
We let $t_u$ be the smallest such $t$ for which
$\bsx_{0,u},\dots,\bsx_{n-1,u}\in[0,1)^{|u|}$
is a $(t,m,|u|)$-net in base $b$.
For theory about $t_u$ see \cite{schm:2001},
for its use defining direction numbers, see \cite{joe:kuo:2008},
and for computational algorithms, see \cite{mari:godi:lecu:2020}.
The quality parameter for the first $b^m$
points of a $(t,s)$-sequence may also be
smaller than the value of $t$ that
holds for the entire sequence.
We will introduce a second quality
parameter for a projected $(t,m,s)$-net
in Section~\ref{sec:reduced}.

\subsection{Scrambling nets}

A scrambled net is one where
the base $b$ digits of a $(t,m,s)$-net
in base $b$ have been randomly permuted
in such a way that the resulting points
satisfy $\bsx_i\sim\dunif[0,1]^s$ individually
while the ensemble $\bsx_0,\dots,\bsx_{n-1}$
is still a $(t,m,s)$-net in base $b$ with probability one.
See \cite{rtms} for the details of a nested uniform scramble and \cite{mato:1998} for
a random linear scramble of Matou\v{s}ek 
that requires less storage.

The nested uniform scrambling has the following properties:
\begin{align*}
&\e( \hat\mu_{\rqmc})=\mu && f\in L^1[0,1]^s,\\
&\Pr\bigl( \lim_{n\to\infty} \hat\mu_{\rqmc}=\mu\bigr)=1&& f\in L^{1+\epsilon}[0,1]^s,\quad \text{some $\epsilon>0$,}\\
&\var( \hat\mu_{\rqmc} )= o(1/n)&& f\in L^2[0,1]^2,\\
&\var(\hat\mu_{\rqmc}) \le \Gamma\sigma^2/n && \var(f(\bsx))=\sigma^2,\quad\text{some $\Gamma<\infty$,}\\
& \var(\hat\mu_{\rqmc}) = O( n^{-3}\log(n)^{s-1}) && \partial^uf\in L^2[0,1]^s\quad
\text{all $u\subseteq1{:}s$},\quad\text{and}\\
& \var(\hat\mu_{\rqmc}) = O( n^{-2}\log(n)^{2(s-1)}) && f\in\bvhk[0,1]^s.
\end{align*}
See \cite{sllnrqmc} for references. It is likely that the random linear
scrambling has these moment properties too.
The rate $O(n^{-3}\log(n)^{s-1})$ is  established under somewhat
weaker conditions than stated above by Yue and Mao \cite{yue:mao:1999}.
There is also a central limit theorem for nested
uniform sampling when $t=0$ due to Loh \cite{loh:2003}.

If $f$ is singular then $f\not\in \bvhk$, so most QMC theory
does not apply to it.
Many singular integrands of interest are in $L^2$, and so
RQMC theory applies to them.
Similarly, step discontinuities and discontinuites in the derivative
of $f$ typically lead to $f\not\in\bvhk$ \cite{variation}
while not ruling out $f\in L^2$.

The constant $\Gamma$ above is the maximum
gain coefficient of the digital net.   It is the
key quantity that we study here.

\subsection{Gain coefficients}

The gain coeffiicents we study are
defined with respect to a different
parameterization of elementary intervals.
For $u\subseteq1{:}s$, $\bsk\in\natu_0^{|u|}$
and $\bsc\in\natu_0^{|u|}$ with $c_j<b^{k_j}$ let
\begin{align}\label{eq:newelint}
E(u,\bsk,\bsc) = \prod_{j\in u}\Bigl[
\frac{c_j}{b^{k_j}}, \frac{c_j+1}{b^{k_j}}\Bigr)
\prod_{j\not\in u}[0,1).
\end{align}
In this representation
$\vol(E(u,\bsk,\bsx)) = b^{-|u|-|\bsk|}$.

Using a base $b$ Haar wavelet decomposition
of $L^2[0,1]^s$ in \cite{owensinum}
we can write $f\in L^2[0,1]^s$ as
$$
f(\bsx) = \sum_{u\subseteq 1{:}s}\sum_{\bsk\in\natu_0^{|u|}}\nu_{u,\bsk}(\bsx)
$$
where the function $\nu_{u,\bsk}$ is constant
within the elementary intervals~\eqref{eq:newelint}.
These functions are defined there in a way that
makes them mutually orthogonal.
For $u=\emptyset$ there is just one
of these functions,
and it is constant over $[0,1]^s$ with $\nu_{\emptyset,()}(\bsx)=\mu$
for all $\bsx$.

From the orthogonality of $\nu_{u,\bsk}$ we find that
$$
\sigma^2\equiv \var(f(\bsx)) =
\sum_{u\ne\emptyset}\sum_{\bsk\in\natu_0^{|u|}}\sigma^2_{u,\bsk}
$$
where for $u\ne\emptyset$ we let
$\sigma^2_{u,\bsk} = \var(\nu_{u,\bsk}(\bsx)) = \int_{[0,1]^s}\nu_{u,\bsk}(\bsx)^2$.
Therefore with plain MC,
$$
\var(\hat\mu_{\mc}) =
\frac1n\sum_{u\ne\emptyset}\sum_{\bsk\in\natu_0^{|u|}}\sigma^2_{u,\bsk}.
$$

If instead of plain MC we use scrambled nets,
then from \cite{owensinum} the sample averages of $\nu_{u,\bsk}$
are still uncorrelated and
$$
\var(\hat\mu_{\rqmc}) =
\frac1n
\sum_{u\ne\emptyset}\sum_{\bsk\in\natu_0^{|u|}}
\Gamma_{u,\bsk}\sigma^2_{u,\bsk}
$$
for gain coefficients $\Gamma_{u,\bsk}$ 
defined at~\eqref{eq:gaineqn} below.
The maximal gain coefficient is
$$
\Gamma = \max_{u\ne\emptyset}\max_{\bsk\in\natu_0^{|u|}} \Gamma_{u,\bsk},
$$
and then $\var(\hat\mu_{\rqmc}) \le
\Gamma\sigma^2/n=\Gamma\var(\hat\mu_{\mc})$.

For a scrambled $(t,m,s)$-net in base $b$,
if $m-t\ge |u|+|\bsk|$, then
all of $E(u,\bsk,\bsc)$ contain the same
number of points of the net.
As a result $\nu_{u,\bsk}$ is integrated without
error and  $\Gamma_{u,\bsk}=0$.

The general formula for gain coefficients
when scrambling points $\bsx_0,\dots,\bsx_{n-1}$ is
\begin{align}\label{eq:gaineqn}
\Gamma_{u,\bsk} = \frac1{n(b-1)^{|u|}}\sum_{i=0}^{n-1}\sum_{i'=0}^{n-1}
\prod_{j\in u} \bigl(
b\one_{\lfloor b^{k_j+1}x_{ij}\rfloor=\lfloor b^{k_j+1}x_{i'j}\rfloor}
-
\one_{\lfloor b^{k_j}x_{ij}\rfloor=\lfloor b^{k_j}x_{i'j}\rfloor}
\bigr).
\end{align}
Here $\one_{\lfloor b^kx_{ij}\rfloor
=\lfloor b^kx_{i'j}\rfloor}$
means that $x_{ij},x_{i'j}\in[0,1)$ match
in their first $k$ base $b$ digits.
The bounds from~\cite{snxs}
are based on equation~\eqref{eq:gaineqn}.
Equation~\eqref{eq:gaineqn} holds for whatever
points we might choose to scramble, not just
digital nets. However, the way digital nets are
constructed tends to
give them small values of $\Gamma_{u,\bsk}$.

When $b=2$, the factors being multiplied
in~\eqref{eq:gaineqn} can
only take three distinct values,
$0$, $-1$, or $1$,
according to whether
$x_{ij}$ and $x_{i'j}$ match
to fewer than $k$ bits,
exactly $k$ bits,
or more than $k$ bits.
Also the factor $(b-1)^{-|u|}$
reduces to $1$.
From this we get the simple bound
\begin{align}\label{eq:gaineqnsobol}
\Gamma_{u,\bsk} \le
\frac1n\sum_{i=0}^{n-1}\sum_{i'=0}^{n-1}
\prod_{j\in u}
\one_{\lfloor b^{k_j}x_{ij}\rfloor=\lfloor b^{k_j}x_{i'j}\rfloor}
\end{align}
when scrambling in base $2$.
We will see below that the bound in~\eqref{eq:gaineqnsobol}
is a power of two.  More surprisingly the exact gain in~\eqref{eq:gaineqn}
is either $0$ or a power of two.

\subsection{Prior gain bounds}

From Lemma 3 in \cite{snxs} we get
\begin{align}\label{eq:oldlemma3}
\Gamma_{u,\bsk}\le b^t\frac{b^{|u|}+(b-2)^{|u|}}{2(b-1)^{|u|}},
\quad\text{when $m-t\le|\bsk|$,}
\end{align}
for a slight generalization of $(t,m,s)$-nets in base $b$.
The statement of that Lemma has
$m-t<|\bsk|$ but the proof technique
also applies when $m-t=|\bsk|$.
In the case $b=2$, the bound simplifies to
$2^{t+|u|-1}$.  In Section~\ref{sec:bound} we
extend this bound to all $\Gamma_{u,\bsk}$.
Lemma 4 of \cite{snxs} gives
$$
\Gamma_{u,\bsk}\le b^t\Bigl(\frac{b+1}{b-1}\Bigr)^{|u|}
\quad\text{when $|\bsk|<m-t<|u|+|\bsk|$.}
$$
It is that Lemma that yields
the bound $\Gamma\le2^t3^s$ for
nets in base $2$.

When $t=0$, \cite{owensinum} shows that
$\Gamma_{u,\bsk}\le (b/(b-1))^{s-1}$.
Because such nets are only possible when $b\ge s$
we get $\Gamma_{u,\bsk}\le (b/(b-1))^{b-1}\le\exp(1)$.  Despite this very
low upper bound on worst case $\var(\hat\mu_{\rqmc})/\var(\hat\mu_{\mc})$,
nets in base $2$ are most used in practice.

Niederreiter and Pirsic \cite{microstruct} improved on the
bounds of \cite{snxs} by looking at 
the microstructure of digital nets. Microstructure
refers to the placement of points within elementary
intervals of volume smaller than $b^{m-t}$.
For example the fact that Sobol' points
have $t_{\{j\}}=0$ is an aspect of their
microstructure.

For $\bsk\in\natu_0^s$ they introduce
$$A(k_1,\dots,k_s)
=\biggl\lceil
\max_{\bsc\in\ints_{b^{\bsk}}}
\log_b\biggl(\,
\sum_{i=0}^{n-1}\one\{\bsx_i\in E(\bsk,\bsc)\}
\biggr)
\biggr\rceil
$$
where the condition on $\bsc$ is interpreted componentwise.
They also use
$$
A_K = \max_{|\bsk|=K}A(k_1,\dots,k_s).
$$

These quantities are well defined whether $\bsx_i$ 
are a $(t,m,s)$-net in base $b$ or not, but they
simplify for nets.
The reference \cite{microstruct} provides several
interesting upper and lower bounds
on $A(\cdot)$ based on the $t$ parameter of a net,
or based on having all $\bsx_i\in \ints_{b^m}/b^m$
or knowing that one or more of the one dimensional
projections of the net has $t_{\{j\}}=0$.

From Proposition 5.1 of \cite{microstruct}
$$
\Gamma_{u,\bsk} \le b^{A_{|\bsk|}}
\frac{b^{|u|}+(b-2)^{|u|}}{2(b-1)^{|u|}}.
$$
This improves upon~\eqref{eq:oldlemma3} by
reducing the lead exponent of $b$ and by
applying to all gain coefficients.
We are most interested in $b=2$ for which their bound yields
$$
\Gamma_{u,\bsk} \le
2^{A_{|\bsk|} +|u|-1}.
$$
Their Theorem 4.1 shows that for $(t,m,s)$-nets
where $t_{\{1\}}=\cdots=t_{\{s\}}=0$
that $A_{|\bsk|}\le t$ when $|\bsk|>m-t$.
Because $\Gamma_{1:s,\bsk}=0$ whenever $|\bsk|\le m-t$
we then get $\Gamma_{u,\bsk}\le 2^{t+|u|-1}$
and hence $\Gamma \le 2^{t+s-1}$ for Sobol' nets.

\subsection{Constructions of nets}

Here we describe the algorithms to construct
digital nets in base $2$, following \cite{joe:kuo:2008}.
We will describe how to compute $2^m$ points
$\bsx_i\in[0,1)^s$ to $m$ bits each.
That is enough to get points that are a $(t,m,s)$-net
in base $2$.
If one is planning to extend the points
from $n=2^m$ to some larger
sample size $n=2^M$, then it is best to use $m=M$.
The points we generate actually belong
to $\{0,1/n,2/n,\dots,(n-1)/n\}^s$.
After scrambling, one  ordinarily adds random offsets
$\bsu_i\simiid\dunif[0,1/n)^s$ to the $\bsx_i$.

A digital net is defined in terms of $s$ matrices
$C_j\in\{0,1\}^{m\times m}$ for $j=1,\dots,s$.
For Sobol' sequences
$$
C_j =
\begin{pmatrix}
1 & v_{2,j,1} & v_{3,j,1} & \cdots & v_{m,j,1}\\
0 & 1 & v_{3,j,2}           & \cdots & v_{m,j,2}\\
0 & 0 & 1                     & \cdots & v_{m,j,3}\\
\vdots & \vdots & \vdots &\ddots & \vdots\\
0 & 0 & 0 & \cdots & 1
\end{pmatrix}
$$
defined in terms of direction numbers $v_{k,j}$ that equal
$0.v_{k,j,1}v_{k,j,2}v_{k,j,3}\dots$
in their base $2$ representation.
Note especially that the matrix $C_j$ is
upper triangular and has $1$s on its diagonal.
Sobol' points ordinarily have $C_1=I_m$.

The digital net construction works as follows.
For $0\le i< 2^m$
write $i=\sum_{\ell=1}^mi_\ell2^{\ell-1}$
for bits $i_\ell\in\{0,1\}$.
Similarly, write $x_{ij} = \sum_{\ell=1}^mx_{ij\ell}2^{-\ell}$
for bits $x_{ij\ell}\in\{0,1\}$.
Then the net $\bsx_0,\dots,\bsx_{2^m-1}$ is
defined by
$$
\begin{pmatrix}
x_{ij1}\\
x_{ij2}\\
\vdots\\
x_{ijm}\\
\end{pmatrix}
= C_j
\begin{pmatrix}
i_{1}\\
i_{2}\\
\vdots\\
i_{m}\\
\end{pmatrix}
\quad\tmod 2.
$$

To define $t$ we
describe a process of forming new
matrices by combining some of the
rows of $C_1,\dots,C_s$.
Let $C_j^{(k)}\in\{0,1\}^{k\times m}$ be
the first $k$ rows of $C_j$.
Then for a non-empty
$u=(r_1,r_2,\dots,r_{|u|})\subseteq1{:}s$
and a vector $\bsk = (k_{r_1},k_{r_2},\dots,k_{r_{|u|}})\in
\{0,1,\dots,m\}^s$, let
$$
C_{u,\bsk} =
\begin{pmatrix}
C_{r_1}^{(k_1)}\\
C_{r_2}^{(k_2)}\\
\vdots\\
C_{r_{|u|}}^{(k_{|u|})}\\
\end{pmatrix}\in\{0,1\}^{|\bsk|\times m}.
$$

The $t$ value of a digital net in base $2$,
constructed from $C_1,\dots,C_m$ is
the smallest value of $t$
such that $C_{u,\bsk}$ has linearly
independent rows over $\ints_2$
whenever $|\bsk|\le m-t$.
This value is the smallest $t$ for which
the definition in terms of $E(\bsk,\bsc)$ 
holds.
The description above applies to any
binary matrices $C_1,\dots,C_m\in\{0,1\}^{m\times m}$,
not just upper triangular ones.


\section{Bound on $\Gamma$}\label{sec:bound}

In this section we prove that $\Gamma_{u,\bsk}\le 2^{t+|u|-1}$.
It follows that $\Gamma \le 2^{t+s-1}$.
We make extensive use of the following 
elementary fact.
\begin{proposition}\label{prop:basic}
Let $A\in\{0,1\}^{K\times m}$ have rank $r$
over $\ints_2$
and let $y\in\{0,1\}^K$.
Then the set of solutions $x\in\{0,1\}^m$
to $Ax=y \tmod 2$
has cardinality $0$ or $2^{m-r}$.
\end{proposition}

We need to keep track of the number of
bits where $x,x'\in[0,1)$ match.
For this we define
$$
M(x,x') = \max\bigl\{k\in\natu_0\mid
\lfloor 2^kx\rfloor
=\lfloor 2^kx'\rfloor\bigr\}\in \natu_0\cup\{\infty\}.
$$
Now for points that are scrambled in base $2$ we get
\begin{align}\label{eq:gain2}
\Gamma_{u,\bsk} &= \frac1n
\sum_{i=0}^{n-1}\sum_{i'=0}^{n-1}\prod_{j\in u}N_{i,i',j}
\end{align}
for
\begin{align*}
N_{i,i',j}  =
\begin{cases}
\phm0, & M(x_{ij},x_{i'j}) < k_j\\
-1, & M(x_{ij},x_{i'j}) =k_j\\
\phm1, & M(x_{ij},x_{i'j}) >k_j.
\end{cases}
\end{align*}

We use arrows to denote bit vectors
derived from values in $[0,1)$ or in $\natu_0$.
For an integer $i=\sum_{\ell=1}^mi_\ell2^{\ell-1}$
with $i_\ell\in\{0,1\}$ we write
$\vec{i}=(i_1,i_2,\dots,i_m)^\tran$
and for $x=\sum_{\ell=1}^mx_\ell2^{-\ell}$
we write $\vec{x}=(x_{1},x_{2},\dots,x_{m})^\tran$.
In either usage, $\vec{0}=(0,0,\dots,0)^\tran$
and there are no nonzero values in
$[0,1)\cap\natu_0$, so the mapping to $\{0,1\}^m$ 
is well defined.
We only need to represent the
bits of $2^m$ integers in $\ints_{2^m}$ and
$2^m$ of the points in $[0,1)$.
Some points in $x\in[0,1)$ have two binary representations,
such as $1/4 = 0.010000\dots=0.001111\cdots$.
We use the choice that ends in a tail of $0$s,
via $x_\ell = \lfloor 2^\ell x\rfloor \tmod 2$.

We will also need to represent some sets
of integers as bit vectors. Given a set $u\subseteq1{:}s$
and $v\subseteq u$, we let $\vec{v}=\vec{v}[u]\in\{0,1\}^{|u|}$
have bits $1$ for  indices corresponding to
elements of $v$
and $0$ for indices corresponding to
elements of $u\setminus v$.

Arithmetic on bit vectors is done componentwise
modulo 2.
We write $\vec{i}\oplus\vec{j}$ and 
$\vec{i}\ominus\vec{j}$ for the componentwise
sum and difference of bit vectors.

For non-empty
$u=\{r_1,\dots,r_{|u|}\}\subseteq1{:}s$ and $\bsk\in\natu_0^{|u|}$ we
define 
$$C_{u,\bsk+\one}=C_{u,\bsk'}\quad\text{where $k'_{r_j}=k_{r_j}+1$ for
$j=1,\dots,|u|$.}$$
Thus $C_{u,\bsk+\one}$ has $|u|$ additional rows in it
beyond those in $C_{u,\bsk}$.
We write the matrix with just these 
$|u|$ additional rows as 
$$\nabla C_{u,\bsk}=
\begin{pmatrix}
C_{r_1}(k_{r_1}+1,:\,)\\
C_{r_2}(k_{r_2}+1,:\,)\\
\vdots\\
C_{r_{|u|}}(k_{r_{|u|}}+1,:\,)
\end{pmatrix}.$$

With the above setup, we are ready to
establish our bounds.
Within the proof of the next
theorem we show that
$$\sum_{i'=0}^{2^m-1}\prod_{j\in u}N_{i,i',j}
=\sum_{i'=0}^{2^m-1}\prod_{j\in u}N_{0,i',j}$$
by symmetry and then bound that sum
using Proposition~\ref{prop:basic}.

\begin{theorem}\label{thm:gainrepresentation}
For integers $m\ge1$ and $s\ge1$,
let $C_1,\dots,C_s\in\{0,1\}^{m\times m}$
generate the digital net $\bsx_0,\dots,\bsx_{2^m-1}$
via $\vec{x}_{ij} = C_j\vec{i}$ for $0\le i<2^m$
and $j=1,\dots,s$.
Then for nonempty $u\subseteq1{:}s$
and $\bsk\in\natu_0^{|u|}$
the gain coefficient $\Gamma_{u,\bsk}$
from~\eqref{eq:gain2} satisfies
\begin{align}\label{eq:gainrepresentation}
\Gamma_{u,\bsk}&=\sum_{i\in\ints_{2^m}}
\one_{C_{u,\bsk}\vec{i}=0}\prod_{j\in u}N_{0,i,j}\notag\\
&=
\sum_{v\subseteq u}\#\bigl\{ i\in\ints_{2^m}
\mid C_{u,\bsk}\vec{i}=0,
\,\nabla C_{u,\bsk}\vec{i}=\vec{v}[u]
\bigr\}(-1)^{|v|}.
\end{align}
\end{theorem}
\begin{proof}
For any $i\in\ints_{2^m}$,
there is some $\bsc$ with $c_j\in\ints_{2^{k_j}}$
for which
$\bsx_i\in E(u,\bsk,\bsc)$.
Then for $i'\in\ints_{2^m}$
with $\bsx_{i'}\not\in E(u,\bsk,\bsc)$
we have $N_{i,i',j}=0$ for some $j\in u$.
As a result, $\prod_{j\in u}N_{i,i',j}=0$ unless $\bsx_{i'}\in E(u,\bsk,\bsc)$ too.
Having both points in the same $E(u,\bsk,\bsc)$
happens if and only if
$C_{u,\bsk}\vec{i}=C_{u,\bsk}\vec{i'}$,
and so only $i'$ with $C_{u,\bsk}(\vec{i'}\ominus\vec{i})=0$
have $\prod_{j\in u}N_{i,i',j}\ne0$.

Now for $\bsx_{i'}\in E(u,\bsk,\bsc)$
it remains to find the sign of  $\prod_{j\in u}N_{i,i',j}$.
In that case
$$N_{i,i',j}=
\begin{cases}
-1, &M(x_{i'j},x_{ij})=k_j\\
\phm1, & M(x_{i'j},x_{ij})>k_j.
\end{cases}
$$
Suppose that $N_{i,i',j} = -1$ for $j\in v\subseteq u$
and $N_{i,i',j}=1$ for $j\in u\setminus v$.
This happens when
and only when $\nabla C_{u,\bsk}(\vec{i'}\ominus\vec{i})
=\vec{v}[u]$.
Then $\prod_{j\in u}N_{i,i',j}=(-1)^{|v|}$.
It follows that
\begin{align*}
\Gamma_{u,\bsk}
&=\frac1n\sum_{i=0}^{2^m-1}\sum_{v\subseteq u}\#
\bigl\{ i'\in\ints_{2^m}
\mid C_{u,\bsk}(\vec{i'}\ominus\vec{i})=0,
\nabla C_{u,\bsk}(\vec{i'}\ominus\vec{i})=\vec{v}[u]
\bigr\}(-1)^{|v|}\\
&=
\sum_{v\subseteq u}\#
\bigl\{ i'\in\ints_{2^m}
\mid C_{u,\bsk}\vec{i'}=0,
\nabla C_{u,\bsk}\vec{i'}=\vec{v}[u]
\bigr\}(-1)^{|v|}
\end{align*}
where the second step
follows because $\vec{i'}\ominus\vec{i}$
runs over the set $\{0,1\}^m$ for any $i\in\ints_{2^m}$.
\end{proof}

The next corollary is already known
from the definition of $(t,m,s)$-nets.
We include it to show how it
follows from Theorem~\ref{thm:gainrepresentation}
and because we need it below.
\begin{corollary}\label{cor:gainzero}
If $C_{u,\bsk+\one}$ has full row rank $|u|+|\bsk|$
for non-empty $u\subseteq1{:}s$,
then $\Gamma_{u,\bsk}=0$.
\end{corollary}
\begin{proof}
When $C_{u,\bsk+1}$ has full row rank
then $C_{u,\bsk}\vec{i}=0$
and $\nabla C_{u,\bsk}\vec{i}=\vec{v}[u]$
has $2^{m-\rank(C_{u,\bsk+1})}$
solutions for all $v\subseteq u$.
Now Theorem~\ref{thm:gainrepresentation} yields
$\Gamma_{u,\bsk}=2^{m-\rank(C_{u,\bsk+1})}
\sum_{v\subseteq u}(-1)^{|v|}=0.$
\end{proof}

\begin{corollary}\label{cor:thebound}
$\Gamma_{u,\bsk}\le 2^{m-\rank(C_{u,\bsk})}$.
\end{corollary}

\begin{proof}
We can rewrite equation~\eqref{eq:gainrepresentation} as
\begin{align*}
\Gamma_{u,\bsk}
&=\sum_{v\subseteq u}\#\bigl\{ \vec{i}\in
\{0,1\}^m\mid C_{u,\bsk}\vec{i}=0,
\nabla C_{u,\bsk}\vec{i}=\vec{v}[u]
\bigr\}(-1)^{|v|}\\
&\le\sum_{v\subseteq u}\#\bigl\{ \vec{i}\in
\{0,1\}^m\mid C_{u,\bsk}\vec{i}=0,
\nabla C_{u,\bsk}\vec{i}=\vec{v}[u]\bigr\}\\
&=\#\bigl\{ \vec{i}\in
\{0,1\}^m\mid C_{u,\bsk}\vec{i}=0\bigr\}\\
&=2^{m-\rank(C_{u,\bsk})}.
\end{align*}
The last step follows from Proposition~\ref{prop:basic} after noting that there is at least one
solution because $\vec{0}$ is a solution.
\end{proof}

\begin{corollary}\label{cor:thebound2}
Let $\bsx_i\in[0,1)^s$ for $i\in\ints_{2^m}$ be
a $(t,m,s)$-net in base $2$.
Then
$$
\Gamma_{u,\bsk} \le 2^{t+|u|-1}.
$$
\end{corollary}

\begin{proof}
If $C_{u,\bsk+1}$ has full row rank then
$\Gamma_{u,\bsk}=0$ by Corollary~\ref{cor:gainzero}.

Suppose next that $C_{u,\bsk}$  has full row
rank but $C_{u,\bsk+1}$ does not.
The matrix $C_{u,\bsk+\one}$ has $|u|+|\bsk|$
rows and this must be at least $m-t+1$
by the definition of $t$ for a $(t,m,s)$-net.
Because $C_{u,\bsk}$ has full row rank
we get $\rank(C_{u,\bsk})=|\bsk|$ and
then by Corollary~\ref{cor:thebound},
$$\Gamma_{u,\bsk}\le 2^{m-\rank(C_{u,\bsk})}
=2^{m-|\bsk|}\le 2^{t+|u|-1}.$$

The remaining case is that $C_{u,\bsk}$ does not
have full row rank. In that case $|\bsk|\ge m-t+1$.
The matrix $C_{u,\bsk}$ must have a subset of $m-t$ rows
defined by $C_{u,\bsk'}$ with $\bsk'\le \bsk$ componentwise
for which $C_{u,\bsk'}$ has full row rank.
Then by Corollary~\ref{cor:thebound},
$$\Gamma_{u,\bsk}\le 2^{m-\rank(C_{u,\bsk})}
\le 2^{m-\rank(C_{u,\bsk'})}=2^{t}\le 2^{t+|u|-1}. \qedhere$$
\end{proof}

Sobol' matrices are upper triangular
with ones on their diagonal. 
Corollary~\ref{cor:thebound2} does not
require upper triangular matrices
or ones on the diagonal. Those properties
are important but their benefit
comes through $t$.

The largest bounds on gain coefficients
come from the case where $C_{u,\bsk}$
has full rank but $C_{u,\bsk+1}$ does not.
Among these, the largest are the ones for
large~$|u|$.

\begin{remark}
The strategy above can be extended
to $(t,m,s)$-nets in base $p$ for prime
numbers $p$ to show that
$\max_{u,\bsk}\Gamma_{u,\bsk}\le p^{t+|u|-1}$.
That will not generally improve on the bound
$p^t((p+1)/(p-1))^{|u|}$ from \cite{snxs}.
Even for $p=3$ it brings improvements only for $|u|\le2$
and raises the bound for $|u|\ge3$.
\end{remark}

\section{$\Gamma$ is a power of $2$}\label{sec:pow2}

Here we prove that the upper bound is actually
tight, so that $\Gamma_{u,\bsk}$ is either $0$
or $2^{m-\rank(C_{u,\bsk})}$, making the maximal gain
a power of $2$ (because it is impossible to have every $\Gamma_{u,\bsk}=0$).
We need some further notation. For $w\subseteq u\subseteq1{:}s$
and $\bsk\in\natu_0^{|u|}$ let
$\bsk+\one_w$ be the vector $\bsk'\in\natu_0^{|u|}$
with $k'_j=k_j+1$ for $j\in w$ and $k'_j=k_j$ for $j\in u\setminus w$.
We then introduce a generalized gain coefficient
\begin{align}\label{eq:gengaincoef}
\Gamma_{u,\bsk}^w = \sum_{i\in\ints_{2^m}}
\one\{C_{u,\bsk}\vec{i}=0\}\prod_{j\in w}N_{0,i,j}.
\end{align}
As usual, an empty product is $1$.
Also the matrix with all the rows of $C_{u,\bsk+\one_w}$
that are not in $C_{u,\bsk}$ is denoted
$\nabla^wC_{u,\bsk}\in\{0,1\}^{|w|\times m}$.

\begin{lemma}\label{lem:gamukwzero}
If $w\ne\emptyset$ and $\rank(C_{u,\bsk+\one_w})-\rank(C_{u,\bsk})=|w|$
then $\Gamma_{u,\bsk}^w=0$.
\end{lemma}
\begin{proof}
Because  $\rank(C_{u,\bsk+\one_w})-\rank(C_{u,\bsk})=|w|$,
the image of $\nabla^wC_{u,\bsk}\vec{i}$ for $\vec{i}\in\{0,1\}^m$
with $C_{u,\bsk}\vec{i}=0$ has rank $|w|$.
But $\nabla^wC_{u,\bsk}\vec{i}\in\{0,1\}^{|w|}$, so the image is the whole
space. Therefore for any $v\subseteq w$ the system of equations
$$
C_{u,\bsk}\vec{i}=0\quad \text{and}\quad\nabla^wC(u,\bsk)\vec{i}=\vec{v}[w]
$$
is consistent and has $2^{m-\rank(C_{u,\bsk+\one_w})}$ solutions.
The rest of the proof is like that in Corollary~\ref{cor:gainzero}.
\end{proof}

\begin{lemma}\label{lem:gamukwnotzero}
If $w\ne\emptyset$ and $\rank(C_{u,\bsk+\one_w})-\rank(C_{u,\bsk})<|w|$
then there exists a nonempty $v\subseteq w$ such that
$\prod_{j\in v}N_{0,i,j}=1$ for any $i\in\ints_{2^m}$ with $C_{u,\bsk}\vec{i}=0$.
\end{lemma}
\begin{proof}
By hypothesis, there 
%
exist coefficients
$a_{j,\ell}$ for $1\le \ell\le k_j$ for $j\in u\setminus w$
and $1\le \ell\le k_j+1$ for $j\in w$ with at least
one $a_{j,k_j+1}=1$ for $j\in w$ such that
\begin{align}\label{eq:itslindependent}
\sum_{j\in u\setminus w}\sum_{\ell=1}^{k_j}a_{j,\ell}C_j(\ell,:\,)\oplus
\sum_{j\in w}\sum_{\ell=1}^{k_j+1}a_{j,\ell}C_j(\ell,:\,)=0
\end{align}
with  $C_j(\ell,:\,)\in\{0,1\}^m$ equal to row $\ell$ of~$C_j$.

We will show that
 $v =\{j\in w\mid a_{j,k_j+1}=1\}$
satisfies the conditions of the Lemma.
To do this we choose any $i\in\ints_{2^m}$ with $C_{u,\bsk}\vec{i}=0$
and let $\vec{e}=\nabla^wC_{u,\bsk}\vec{i}\in\{0,1\}^{|w|}$.
Multiplying both sides of \eqref{eq:itslindependent}
by  $\vec{i}$ gives
$$
\sum_{j\in w}a_{j,k_j+1}C_j(k_j+1,:\,) \vec{i}
=\sum_{j\in v} e_j=0 \tmod 2.
$$
Because the bits of $\vec{e}$ sum to zero in $\ints_2$
there must be an even number of them that equal $1$.
There are then an even number of $j\in v$
with $N_{0,i,j}=-1$ and then $\prod_{j\in v}N_{0,i,j}=1$.
\end{proof}

\begin{theorem}
$\Gamma_{u,\bsk}$ is either $0$ or $2^{m-\rank(C_{u,\bsk})}$.
\end{theorem}
\begin{proof}
We proceed by induction on $|w|$
to prove that $\Gamma^w_{u,\bsk}\in\{0,2^{m-\rank(C_{u,\bsk})}\}$
for all $w\subseteq u$.
The conclusion then follows because $\Gamma^u_{u,\bsk}=\Gamma_{u,\bsk}$.

We begin with $w=\emptyset$.
Then from the definition~\eqref{eq:gengaincoef} of generalized gain coefficients,
$$
\Gamma^\emptyset_{u,\bsk}
=\sum_{i\in\ints_{2^m}}\one_{C_{u,\bsk}\vec{i}=0}\prod_{j\in \emptyset}N_{0,i,j}
=\sum_{i\in\ints_{2^m}}\one_{C_{u,\bsk}\vec{i}=0}
\in\{0,2^{m-\rank(C_{u,\bsk})}\}
$$
by Proposition~\ref{prop:basic}.

Now we suppose that $\Gamma^w_{u,\bsk}\in
\{0,2^{m-\rank(C_{u,\bsk})}\}$
holds whenever $0\le|w|< r$ for some $r\le|u|$.
If $|w|=r$ and $\rank(C_{u,\bsk+\one_w})-\rank(C_{u,\bsk})=|w|$
then $\Gamma^w_{u,\bsk}=0$ for $w\ne\emptyset$
by Lemma~\ref{lem:gamukwzero}.

It remains to consider the case with $|w|=r$
and $\rank(C_{u,\bsk+\one_w})-\rank(C_{u,\bsk})<|w|$.
In this case $w$ is not empty and we let $v$ be
the non-empty subset of $w$ from Lemma~\ref{lem:gamukwnotzero}.
Then because $\prod_{j\in v}N_{0,i,j}=1$
$$
\Gamma_{u,\bsk}^w
=\sum_{i\in\ints_{2^m}}\one_{C_{u,\bsk}\vec{i}=0}\prod_{j\in w}N_{0,i,j}
=\sum_{i\in\ints_{2^m}}\one_{C_{u,\bsk}\vec{i}=0}\prod_{j\in w\setminus v}N_{0,i,j}
=\Gamma_{u,\bsk}^{w\setminus v}.
$$
Now $|w\setminus v|<|w|=r$
so we can apply the induction hypothesis.
\end{proof}

Since there are only two possibilities for $\Gamma_{u,\bsk}$
we are able to get a computationally advantageous
check for which of them holds.
\begin{corollary}\label{cor:the check}
$\Gamma_{u,\bsk}=2^{m-\rank(C_{u,\bsk})}$ if
and only if $\sum_{j\in u}C_j(k_j+1,:\,)\in\{0,1\}^m$
is in the row space of $C_{u,\bsk}$.
\end{corollary}
\begin{proof}
First suppose that $\sum_{j\in u}C_j(k_j+1,:\,)$ is in
the row space of $C_{u,\bsk}$.
Then we can apply Lemma~\ref{lem:gamukwnotzero} with $v=w=u$
to get that $\prod_{j\in u}N_{0,i,j}=1$ for any $i\in\ints_{2^m}$ with $C_{u,\bsk}\vec{i}=0$.
Conversely, suppose that $\Gamma_{u,\bsk}=2^{m-\rank(C_{u,\bsk})}$.
Then from details of the proof of Corollary~\ref{cor:thebound},
\begin{align*}
&\sum_{v\subseteq u}\#\bigl\{ \vec{i}\in
\{0,1\}^m\mid C_{u,\bsk}\vec{i}=0,
\nabla C_{u,\bsk}\vec{i}=\vec{v}[u]
\bigr\}(-1)^{|v|}\\
&=\sum_{v\subseteq u}\#\bigl\{ \vec{i}\in
\{0,1\}^m\mid
C_{u,\bsk}\vec{i}=0, \nabla C_{u,\bsk}\vec{i}=\vec{v}[u]
\bigr\},
\end{align*}
which rules out having any solutions $\vec{i}\in\{0,1\}^m$ to
$$
C_{u,\bsk}\vec{i}=0, \nabla C_{u,\bsk}\vec{i}=\vec{v}[u]
$$
for any $v$ with an odd cardinality.
Therefore
$\sum_{j\in u}C_j(k_j+1,:\,)\vec{i}=0$
whenever $C_{u,\bsk}\vec{i}=0$ which then implies
that $\sum_{j\in u}C_j(k_j+1,:\,)$ is in the row space of $C_{u,\bsk}$.
\end{proof}

\section{Reduced upper bound}\label{sec:reduced}
From Corollary~\ref{cor:thebound} we have
$\Gamma_{u,\bsk}\le 2^{t+|u|-1}$.
It follows immediately that $\Gamma_{u,\bsk}\le 2^{t_u+|u|-1}$ because we could have formed
a net out of only $\bsx_{i,u}\subset[0,1)^{|u|}$.
In this section we show that it is possible
to improve that bound.


We need to use a second notion of the $t$
parameter specific to a subset $u\ne\emptyset$
of components of $\bsx_i$.
This notion is denoted $t^*_u$. We define it
below side by side with the prior $t_u$ to
make comparisons easier. We also need a
quantity $t_d$ for $1\le d\le s$ to describe
the quality of projected nets.  The
new and old quantities are
\begin{align*}
t &=
m+1-\min_{u\neq \emptyset, \bsk\in\natu_0^{|u|}}\bigl\{ |\bsk|
\mid
C_{u,\bsk}\ \mathrm{not\ of\ full\ rank}\, \bigr\},\\
t_d &=
m+1-\min_{u:|u|\leq d, \bsk\in\natu_0^{|u|}}\bigl\{ |\bsk|
\mid
C_{u,\bsk}\ \mathrm{not\ of\ full\ rank}\, \bigr\},\\
t_u &=
m+1-\min_{\bsk\in\natu_0^{|u|}}\bigl\{ |\bsk|
\mid
C_{u,\bsk}\ \mathrm{not\ of\ full\ rank}\, \bigr\},\quad\text{and}\\
t^*_u &=
m+1-\min_{\bsk\in\natu_0^{|u|}}\bigl\{ |\bsk|
\mid
C_{u,\bsk}\ \mathrm{not\ of\ full\ rank,\  \bsk\geq \one_u\ \mathrm{componentwise}}\, 
\bigr\}.
\end{align*}
Because $t_u^*$ adds constraints $k_j\geq 1$ for $j\in u$, we have $t_u^*\leq t_u$. 
Likewise, if $v\subseteq u$, then $t_v^*$ adds constraints $k_j\geq 1$ for $j\in v$ and $k_j=0$ for $j\in u \setminus v$, and we have $t_v^*\leq t_u$.
For any $\bsk$ that attains the minimum defined in $t_u$, define $v=\{j\in u\mid k_j\geq 1\}$ and $\bsk'\in \natu_0^{|v|}$ to be the nonzero entries of $\bsk$. Then $C_{u,\bsk}$ is also $C_{v,\bsk'}$ and $t_v^*\geq t_u$.
Therefore
$$t_u=\max_{v\subseteq u}t^*_v.$$ 
From the definitions of $t_d$ and $t$, we have
\begin{align*}
t_d&=\max_{u:|u|\leq d}t_u=\max_{u:|u|\le d}\max_{v\subseteq u}t^*_v=\max_{|v|\le d}t^*_v,\quad\text{and}\\
t&=\max_{1\le d\le s} t_d=\max_{u\neq \emptyset}t_u=\max_{v\neq \emptyset}t^*_v.
\end{align*}
Theorem~\ref{thm:t_u} below shows that
we can replace the bound $2^{t_u+|u|-1}$
by $2^{t_u^*+|u|-1}$.
It then follows that
$$\max_{u:|u|\leq d}\max_{\bsk\in\natu_0^{|u|}}\Gamma_{u,\bsk}\leq \max_{u:|u|\leq d}2^{t^*_u+|u|-1}\leq 2^{t_d+d-1}.
$$

For the next results we need to use the vector
$\one_u = (1,1,\dots,1)\in\ints^{|u|}$.

\begin{theorem}\label{thm:t_u}
For any $u\subseteq1{:}s$ where $C_{u,\one_u}$ has full row rank,
$$
\max_{\bsk\in\natu_0^{|u|}}\Gamma_{u,\bsk}\leq 2^{t^*_u+|u|-1}
\quad\text{and}\quad
\max_{v\subseteq u}\max_{\bsk\in\natu_0^{|v|}}\Gamma_{v,\bsk}= 2^{t^*_u+|u|-1}.
$$
\end{theorem}
\begin{proof}
First we prove that $\Gamma_{u,\bsk}\leq 2^{t^*_u+|u|-1}$. By the definition of $t^*_u$, if a matrix $C_{u,\bsk}$ with $\bsk\ge \one_u$ does not have full row rank, it must satisfy $|\bsk|\geq m+1-t^*_u$. The proof in the case where $C_{u,\bsk}$ has full row rank is like that in Corollary~\ref{cor:thebound2}. The remaining case is when $C_{u,\bsk}$ does not have full row rank.

Define $v=\{j\in u\mid k_j=0\}$. Then $|\bsk|+|v|\ge m-t^*_u+1$ because $C_{u,\bsk+\one_v}$ does not have full row rank and $\bsk+\one_v\ge \one_u$. 
The matrix $C_{u,\bsk+\one_v}$ must have a subset of $m-t^*_u$ rows
defined by $C_{u,\bsk'}$ with $\bsk'\le \bsk+\one_v$
for which $C_{u,\bsk'}$ has full row rank.
Then by Corollary~\ref{cor:thebound},
$$\Gamma_{u,\bsk}\le 2^{m-\rank(C_{u,\bsk})}\le 2^{m-\rank(C_{u,\bsk+\one_v})+|v|}\le
2^{t^*_u+|v|}\le 2^{t^*_u+|u|-1},$$
where the last inequality follows from $v$ being 
a proper subset of $u$ because $\bsk$ cannot be 0.

To prove the second statement, notice that for any $v\subseteq u$ and $\bsk\in\natu_0^{|v|}$ such that $\bsk\geq \one_v$ and $C_{v,\bsk}$ is row rank deficient, we can define $\bsk'\in\natu_0^{|u|}$ so that $k'_j=k_j$ for $j\in v$ and $k'_j=1$ for $j\in u\setminus v$. Then $\bsk'\geq \one_u$ and $C_{u,\bsk'}$ is row rank deficient as well because it is made of $C_{v,\bsk}$ with $|u|-|v|$ extra rows. It follows that $t_v^*\leq t_u^*+|u|-|v|$ and 
$$
\max_{\bsk\in\natu_0^{|v|}}\Gamma_{v,\bsk}\leq 2^{t^*_v+|v|-1}\leq  2^{t^*_u+|u|-1}.
$$

It remains to show that there exists $v\subseteq u$ and $\bsk^v\in\natu_0^{|v|}$ such that $\bsk^v\geq \one_v$ and $\Gamma_{v,\bsk^v}=2^{t^*_u+|u|-1}$. First we choose any $\bsk^*$ that attains the minimum $|\bsk|$ defined in $t^*_u$. Because $C_{u,\bsk^*}$ is not of full rank, its row vectors must be linearly dependent. That is, there exist coefficients
$a_{j,\ell}\in\{0,1\}$ for $1\le \ell\le k^*_j$ and $j\in u$
such that
\begin{align*}
\sum_{j\in u}\sum_{\ell=1}^{k^*_j}a_{j,\ell}C_j(\ell,:\,)=0\tmod 2.            
\end{align*}
Define $v=\{j\in u\mid a_{j,k^*_j}=1\}$ and $\bsk^v\in\natu_0^{|v|}$ such that $k_j^v=k^*_j$ for $j\in v$. Because $\bsk^*$ attains the smallest $|\bsk|$ among $\bsk\geq \one_u$, $a_{j,k^*_j}=1$ for all $j\in u$ such that $k_j^*\geq 2$. In other words, $j\in u\setminus v$ if and only if $k^*_j=1$ and $a_{j,1}=0$. Hence $|\bsk^v|=|\bsk^*|-|u\setminus v|$ and
\begin{align*}
\sum_{j\in v}C_j(k^v_j,:\,)+\sum_{j\in v}\sum_{\ell=1}^{k^v_j-1}a_{j,\ell}C_j(\ell,:\,)=0.
\end{align*}
Corollary~\ref{cor:the check} then implies
that $\Gamma_{v,\bsk^v-\one_v}=2^{m-\rank(C_{v,\bsk^v-\one_v})}$. 
Again, because $\bsk^*$ attains the smallest $|\bsk|$, $C_{v,\bsk^v-\one_v}$ has full row rank and $\rank(C_{v,\bsk^v-\one_v})=|\bsk^v|-|v|=|\bsk^*|-|u|$. 
Therefore
$$\Gamma_{v,\bsk^v-\one_v}=2^{m-\rank(C_{v,\bsk^v-\one_v})}=2^{m-|\bsk^*|+|u|}=2^{t^*_u+|u|-1}.\qedhere$$
\end{proof}

\begin{corollary}
If there exists $u\subseteq1{:}s$
with $\rank(C_{u,\one_u})<|u|$
then the maximal gain coefficient $\Gamma=2^m$.
Otherwise $\Gamma= 2^{t^*_{1:s}+s-1}$.
\end{corollary}
\begin{proof}
In the former case, we can choose the smallest size $u$ whose $C_{u,\one_u}$ is not full rank. Then the row vectors of $C_{u,\one_u}$ are linearly dependent, but any proper subset of them are linearly independent. Therefore $\sum_{j\in u} C_{j}(1,:\,)=0$. We can then apply Corollary~\ref{cor:the check} to $\Gamma_{u,0}$ and derive $\Gamma_{u,0}=2^m$, which is the largest $\Gamma_{u,\bsk}$ in view of Corollary~\ref{cor:thebound}.

In the latter case, the conclusion immediately follows by applying Theorem~\ref{thm:t_u} to $u=1{:}s$.
\end{proof}



\begin{example}
Here we consider a $(1,1,4)$-net
in base $2$, known as a shift net \cite{schm:1998}
because the columns of the first
generator matrix $C_1$ are shifted
to create the other generator matrices.
The top three rows of the generator
matrices $C_1$ through $C_4$ are
\begin{align*}
\begin{pmatrix}
0 &0 &0 &1\\
0 &1 &1& 0\\
0 &0 &1& 0\\
\end{pmatrix},
\begin{pmatrix}
0 &0 &1 &0\\
1 &1 &0& 0\\
0 &1 &0& 0\\
\end{pmatrix},
\begin{pmatrix}
0 &1 &0 &0\\
1 &0 &0& 1\\
1 &0 &0& 0\\
\end{pmatrix}\ \text{and}\ 
\begin{pmatrix}
1 &0 &0 &0\\
0 &0 &1 &1\\
0 &0 &0 &1\\
\end{pmatrix}.
\end{align*}
The fourth rows could be anything
without changing this example.
From $m=s=4$ and $t=1$ we get a gain
coefficient bound $\Gamma\le 2^{t+s-1}=16$.
However $t^*_{1{:}s}=0$ for this net
after observing that
$$
C_{1:4,\one_{1:4}}=\begin{pmatrix}
0 & 0 & 0 & 1\\
0 & 0 & 1 & 0\\
0 & 1 & 0 & 0 \\
1 & 0 & 0 & 0
\end{pmatrix}
$$
has full rank.  Therefore $\Gamma=2^{t^*_{1:s}+s-1}=8$. 
\end{example}
\section{Discussion}\label{sec:disc}

Our first contribution is to tighten the bounds on $\Gamma_{u,\bsk}$ and
hence also
the maximal gain $\Gamma = \max_{|u|>0}\max_{\bsk\in\natu_0^{|u|}}\Gamma_{u,\bsk}
$ for digital nets in base $2$ of which
the constructions of Sobol' \cite{sobo:1967}
and Niederreiter-Xing
\cite{niedxing96, niedxing96b}
are the most important.
Our second contribution is to show that gain coefficients for
base $2$ digital nets must be either $0$ or a power of $2$,
so the maximal correlation is a power of $2$.
Finally, a consequence of our results is a more efficient
algorithm for computing gain coefficients.


\section*{Acknowledgments}
This work was supported by the U.S.\ National Science Foundation under grant IIS-1837931.
\bibliographystyle{plain}
\bibliography{qmc}
\end{document}